\documentclass[11pt]{article}
\usepackage[utf8]{inputenc}
\usepackage{amsthm}
\usepackage{graphicx}
\usepackage{amssymb}
\usepackage{color}
\usepackage[dvipsnames]{xcolor}
\usepackage{tikz}
\usepackage{tikz-cd}
\usetikzlibrary{shapes.geometric}
\usepackage[many]{tcolorbox} 
\usepackage{amsmath}
\usepackage[hidelinks]{hyperref}
\usepackage{float}
\usepackage{mathtools}
\usepackage{enumitem}
\usepackage{geometry}
\usetikzlibrary{arrows.meta}

\usepackage[toc,page]{appendix}

\usepackage{bbm}

\usepackage{pgfplots}
\usepgfplotslibrary{fillbetween}
\pgfplotsset{compat=1.12}

\usepackage[style=numeric, giveninits=true, sortcites=true, backend=biber]{biblatex} 
\AtEveryBibitem{\clearfield{isbn}}
\AtEveryBibitem{\clearfield{issn}}
\AtEveryBibitem{\clearfield{doi}}
\AtEveryBibitem{\clearfield{url}}
\AtEveryBibitem{\clearfield{number}}

\DeclareNameAlias{default}{given-family}

\DefineBibliographyExtras{english}{%
}

\renewbibmacro{in:}{
}

\DeclareFieldFormat*{title}{\mkbibemph{#1}}

\DeclareFieldFormat{pages}{#1}

\DeclareFieldFormat*{journaltitle}{\textnormal{#1}}

\DeclareFieldFormat*{volume}{\mkbibbold{#1}}

\usepackage{wrapfig}

\usepackage{minitoc}

\usepackage{esint}

\theoremstyle{plain}
\newtheorem{theorem}{Theorem}[section]
\newtheorem{lemma}[theorem]{Lemma}

\newtheorem{proposition}[theorem]{Proposition}
\newtheorem{corollary}[theorem]{Corollary}
\newtheorem{claim}[theorem]{Claim}

\newtheorem{problem}[theorem]{Problem}

\theoremstyle{definition}
\newtheorem{definition}[theorem]{Definition}
\newtheorem{exam}[theorem]{Example}

\theoremstyle{remark}
\newtheorem{remark}[theorem]{Remark}

\theoremstyle{plain}
\newtheorem*{theorem*}{Theorem}
\newtheorem*{lemma*}{Lemma}

\def\beq{\begin{equation}}
\def\eeq{\end{equation}}
\def\baeq{\begin{aligned}}
\def\eaeq{\end{aligned}}
\def\lb{\label}

\def\bremark{\begin{remark}}
\def\eremark{\end{remark}}

\newcommand{\RR}{\mathbb{R}}

\newcommand{\CC}{{\mathbb C}}
\newcommand{\ZZ}{{\mathbb Z}}
\newcommand{\NN}{{\mathbb N}}

\newcommand{\dcal}{\mathcal{D}}
\newcommand{\fcal}{\mathcal{F}}

\newcommand{\pcal}{\mathcal{P}}
\newcommand{\scal}{\mathcal{S}}
\newcommand{\tcal}{\mathcal{T}}

\newcommand{\diag}{{\operatorname{diag}}}

\newcommand{\vol}{{\operatorname{Vol}}}

\def\Re{{\operatorname{Re}}}

\def\tr{\operatorname{tr}}

\def\path{\operatorname{Path}}

\def\Sym{\operatorname{Sym}}
\def\i{\sqrt{-1}}

\def\beq{\begin{equation}}
\def\eeq{\end{equation}}
\def\baeq{\begin{aligned}}
\def\eaeq{\end{aligned}}
\def\lb{\label}

\title{Convexity and the degenerate special Lagrangian equation}
\date{\today}
\author{Vasanth Pidaparthy, Yanir A. Rubinstein}

\begin{document}

\maketitle

\begin{abstract}
In 2015 Rubinstein--Solomon introduced the degenerate special Lagrangian equation (DSL) that governs geodesics in the space of positive Lagrangians, showed that 
subsolutions in the top branch of DSL are convex in space, and raised
the question of whether they should be convex in space-time
and whether subsolutions in the second branch possess any convexity
properties.
In 2019, Darvas--Rubinstein gave a partial answer to the first problem by showing subsolutions in the top branch must
be bi-convex.
We settle both questions. 
The key new ingredient is a space-time coordinate transformation that
preserves the space-time Lagrangian angle and allows for a partial $C^2$ estimate. This also shows that the top two branches of the DSL subequation have a $\star$-product structure in the sense of Ross--Witt-Nystr\"om.
\end{abstract}


\section{Motivation}
\label{section:introduction-DSL-convexity-short}

While the potential theory of special Lagrangian equation (SL) has
been much-studied since the seminal works of Harvey--Lawson \cite{HarLaw-calibratedgeom, HarLaw-dirichletduality}, 
the analogous theory for the degenerate special Lagrangian equation
is still quite new and undeveloped. 
This article contributes to the potential theory of the degenerate special Lagrangian (DSL) equation. Our main goal is to obtain the optimal convexity properties of viscosity subsolutions of the DSL. 
These properties turn out to be different than the corresponding
convexity properties of subsolutions of the SL.
This asymmetry between the two equations is unlike the classical situation 
for the Monge--Amp\`ere equation and the degenerate Monge--Amp\`ere equation, and indeed the 
analysis in our setting is more delicate than that in the Monge--Amp\`ere setting.

\subsection{The special Lagrangian equation}

To put the problems we study concerning the DSL in perspective, let us start by quickly reviewing the by-now-classical convexity properties
for subsolutions of the SL. The SL is a second order fully non-linear elliptic PDE on a domain $D \subset \RR^n$
that is the  local model for volume-minimizing
Lagrangian submanifolds. 
A special Lagrangian $L$ is a (calibrated) Lagrangian submanifold of a Calabi--Yau manifold $(X,\omega, J, \Omega)$, satisfying 
$\Omega|_{L}/d\vol_L = \hbox{\rm const}$,
where $\Omega$ denotes a non-vanishing holomorphic $(n,0)$-form
on $X$ and $d\vol_L$ is the Riemannian volume form on $L$ induced from
the Riemannian structure $g(\,\cdot\,,\,\cdot\,):=\omega(\,\cdot\,,J\,\cdot\,)$ restricted to $L$.
While $\Omega$ can be rescaled so that the constant above is 
simply 1, the convention is to fix $\Omega$ once and for all. Then the constant can be any unit complex number $e^{\i \theta}$, where $\theta\in [0,2\pi)$ is the so-called Lagrangian angle of $L$.
 Harvey--Lawson observed that for the model Calabi--Yau manifold $\CC^n$, the graph of the gradient $Dv$ of a function $v:D \rightarrow \RR$ over a domain $D \subset \RR^n$ is 
 SL if and only if it satisfies the SL \cite[\S III.2]{HarLaw-calibratedgeom}
\begin{equation}
    \label{eq:the-SL-DSL-convexity-short}
    \begin{aligned}
        \sum_{i=1}^n \tan^{-1} \lambda_i(D^2v) = c \qquad &\text{on } D,
    \end{aligned}
\end{equation}
for some $c\in \left(-\frac{n}{2}\pi, \frac{n}{2}\pi \right)$ satisfying $c \equiv_{2\pi} 
\theta $, where $\tan^{-1}$ in \eqref{eq:the-SL-DSL-convexity-short} denotes the odd branch 
\beq
\lb{taninvEq}\tan^{-1}:\RR \rightarrow (-{\pi}/{2}, {\pi}/{2}).
\eeq 
The left hand side of \eqref{eq:the-SL-DSL-convexity-short} is called the lifted Lagrangian angle of the Lagrangian graph. For a given $\theta$, 
there are different
choices of $c\in \left(-\frac{n}{2}\pi, \frac{n}{2}\pi \right)$ with $c\equiv_{2\pi} 
\theta $; each corresponds to a different {\it branch} of the SL. Uniqueness of solutions for  \eqref{eq:the-SL-DSL-convexity-short} cannot be expected (and, in fact, fails) unless a particular branch is chosen, and in general the regularity of solutions
varies greatly depending on the choice of the branch; as a rule, solutions in higher branches possess better regularity.
According to Harvey--Lawson there is always a viscosity solution to the Dirichlet problem for \eqref{eq:the-SL-DSL-convexity-short} in each branch for suitable Dirichlet data. Much work has gone into showing regularity properties for such solutions, and essentially all known results are for the {\it top two branches}, i.e.,
for $c\in \left[\frac{n-2}{2}\pi,\frac{n-1}{2}\pi \right)$ (second branch) and $c\in \left[\frac{n-1}{2}\pi,\frac{n}{2}\pi \right)$ (top branch).

The potential theory for special Lagrangians has been intensely studied. Caffarelli--Nirenberg--Spruck \cite{caffarelli-nirenberg-spruck-dirichlet-problem-III-MR0806416} first studied the Dirichlet problems for \eqref{eq:the-SL-DSL-convexity-short} with smooth boundary data when $c= \frac{n-1}2\pi$ or $c =\frac{n-2}2\pi$, and constructed $C^2$ solutions when the domain $D$ is uniformly convex. Their method can 
be extended to all angles in the outer two branches $|c| \ge \frac{n-2}2\pi$ constructing $C^2$ solutions in these branches \cite{Trudinger1995-DHE, collins-picard-wu-MR3659638}. Harvey--Lawson constructed continuous viscosity solutions to the Dirichlet problem in every branch by an upper envelope method under convexity assumptions on the domain $D$ \cite[Corollary~10.5]{HarLaw-dirichletduality}. However only solutions in the top two branches $|c| \ge \frac{n-2}2\pi$ are $C^{1,1}$ (everything that can be said about the top two branches also carries over to the bottom two branches, so
we ignore these from now on). In fact, counter-examples to $C^{1,\delta}$ regularity are known in all the inner branches $|c| \le \frac{n-2}2\pi$ for every $\delta>0$ by Nadirashvili--Vl\u{a}du\c{t} \cite{nadirashvili-Vladut-MR2683755} and Wang--Yuan \cite{wang-yuan-singular-solutions-MR3117304}.

Convexity of subsolutions to the SL \eqref{eq:the-SL-DSL-convexity-short} is also well understood. Recall that an upper semi-continuous $v:D \rightarrow \RR$ is called a subsolution to \eqref{eq:the-SL-DSL-convexity-short} in branch $c$ if
\beq
    \label{eq:the-SL-DSL-convexity-short-subsol}
\sum_{i=1}^n \tan^{-1} \lambda_i (D^2v) \ge c,
\eeq
in the viscosity sense. The set of all viscosity subsolutions in the branch $c$ is denoted $F_c(D)$. It follows
directly from 
\eqref{taninvEq}
and \eqref{eq:the-SL-DSL-convexity-short-subsol}
that subsolutions in the top branch $c\ge \frac{n-1}2\pi$ 
are convex, 
All other branches admit subsolutions which are 
not convex.
Subsolutions in the second branch are known to be 2-convex, i.e., the sum of any two eigenvalues of the Hessian is positive by an observation of Yuan \cite{yuan-global-solutions-MR2199179}
(see Lemma~\ref{lemma:properties-of-eigenvalues-DSL-convexity-short}).

\subsection{The degenerate special Lagrangian equation}

Rubinstein--Solomon introduced the 
degenerate special Lagrangian equation (DSL)
and studied the 
Dirichlet problem for the DSL as a degenerate analog of the SL \cite[Theorem 4.1]{Rub-Sol-DSL-MR3620699}. The DSL is fundamental in Lagrangian geometry since it governs geodesics of positive Lagrangian graphs in $\CC^n$. It is a degenerate elliptic PDE for a function $u:(0,1)\times D \rightarrow \RR$, for a domain $D\subset \RR^n$:
\begin{equation}
    \label{eq:the-DSL-DSL-convexity-short}
    \begin{aligned}
        \widetilde{\Theta}(D^2u) = c \qquad &\text{on } \dcal:= (0,1) \times D,
    \end{aligned}
\end{equation}
where 
$\widetilde{\Theta}: \Sym^2(\RR^{n+1}) \rightarrow \RR$ is the upper semi-continuous \textit{lifted space-time Lagrangian angle} (Definition~\ref{def:spacetimelagrangianangle-DSL-convexity-short}), and $c\in \left(-\frac{n+1}2\pi, \frac{n+1}2\pi \right)$.
Rubinstein--Solomon proved existence of Lipschitz continuous viscosity solutions to the Dirichlet problem for \eqref{eq:the-DSL-DSL-convexity-short} in every branch of the DSL for strictly convex domains $D$ and suitable boundary conditions \cite[Theorem 1.2]{Rub-Sol-DSL-MR3620699}. 

An upper semi-continuous $u:\dcal\rightarrow \RR$ is a (viscosity) subsolution to the DSL in branch $c$ if $\widetilde{\Theta}(D^2u) \ge c$ in the viscosity sense, and the set of such subsolutions is denoted by $\fcal_c(\dcal)$. Rubinstein--Solomon observed that the restriction to a time slice $u(t,\cdot)$ of a viscosity subsolution $u\in\fcal_c(\dcal)$ in the top branch $c \ge \frac{n}{2}\pi$ is a subsolution to the top branch of the SL,
i.e., $u(t,\,\cdot\,)\in F_{c-\pi/2}$  (for $t\in(0,1)$). In particular it is convex, making $u$ convex in space \cite[Corollary~9.2]{Rub-Sol-DSL-MR3620699}. It was later observed by Darvas--Rubinstein that subsolutions in the top branch are also convex in time \cite[Lemma~3.6]{Dar-Rub2019AMP}. Thus, such subsolutions are
{\it bi-convex}, i.e., separately convex in space and in time.

\begin{problem}
    \label{problem:joint-convexity-DSL-convexity-short}
    Are subsolutions to the top branch of the DSL \eqref{eq:the-DSL-DSL-convexity-short}, i.e., for $c\in \left[\frac{n}{2}\pi, \frac{n+1}2\pi \right)$, jointly convex on space-time?
\end{problem}

Joint convexity is a considerable strengthening of bi-convexity, and in particular implies that mixed space-time second derivatives are also bounded
a.e., or geometrically, the tangent space to the space-time Lagrangian graph exists a.e.

Analogous to the SL, it is expected that solutions to the outer two branches of the DSL should enjoy better regularity than solutions in the inner branches. Motivated by this, we also investigate convexity of subsolutions in the second branch of the DSL. However there are plenty of examples of solutions and subsolutions in the second branch which fail to be convex in space, and are generally not 2-convex either. Recall that a function is 2-convex if it is subharmonic along every two-dimensional plane. If the function is $C^2$ then 2-convexity is equivalent to non-negativity of the sum of any pair of Hessian eigenvalues.

\begin{exam}
\label{example:second-branch-DSL-convexity-short}
    Fix $\eta>0$ and $\delta\in (0,\frac{\pi}{2})$. Then for any $\epsilon \in (0, \frac{\pi}{2} - \delta)$ sufficiently small let 
    \[\theta_0:= \frac{\pi}{2} - \frac{\epsilon}{n-1}>0, \qquad 
    \qquad \theta_1:= \epsilon + \delta - \frac{\pi}{2} <0,\]
    and consider the quadratic
    \[u(t,x):= \frac{1}{2} \left( \eta t^2 + \sum_{i=1}^{n-1} x_i^2\tan\theta_0 + x_n^2\tan\theta_1  \right). \]
    Since $u$ is a quadratic function, observe that
    \[\begin{aligned}
        &\ \arg\left( I_n + \sqrt{-1}D^2u \right)  \\
        &\ = \arg \left(\begin{pmatrix}
        0  & & & &\\
         & 1 & &  &\\
          & & \ddots& & \\
           & & &  1 &\\
           & & & &  1
    \end{pmatrix} + \sqrt{-1}\begin{pmatrix}
        \eta  & & & &\\
         & \tan \theta_0 & &  &\\
          & & \ddots& & \\
           & & &  \tan \theta_0 &\\
           & & & &  \tan \theta_1
    \end{pmatrix} \right)\\
    &\ = 
    \begin{pmatrix}
        \frac{\pi}{2} & & & &\\
         &  \theta_0 & &  &\\
          & & \ddots& & \\
           & & &  \theta_0 &\\
           & & & &  \theta_1
    \end{pmatrix},
    \end{aligned}\]
    so
    $\widetilde{\Theta}(D^2u) = \frac{\pi}{2} + (n-1)\theta_0 + \theta_1 \ge \frac{n-1}2\pi + \delta
    $ 
    by \eqref{eq:lifted-spacetime-lag-angle-DSL-convexity-short}, i.e., $u \in \fcal_{\frac{n-1}2\pi + \delta}(\RR^{n+1})$. Yet, $\partial_{x_n}^2 u = \tan( \theta_1) = \tan\left(\epsilon + \delta - \frac{\pi}{2} \right) <0$, so $u$ is not convex in $x$. In fact, for $\eta$ sufficiently small, $u$ will not even be 2-convex since $D^2u$ has eigenvalues $\eta$ and $\tan\left( \epsilon + \delta - \frac{\pi}{2} \right)$, and for $\epsilon, \delta, \eta$ sufficiently small
    \[\eta + \tan\left( \epsilon + \delta - \frac{\pi}{2} \right) <0 .\]
\end{exam}
The natural question extending Problem~\ref{problem:joint-convexity-DSL-convexity-short} to the second branch is the following:

\begin{problem}
    \label{problem:joint-convexity-2nd-branch-DSL-convexity-short}
    Are subsolutions to the top two branches of the DSL \eqref{eq:the-DSL-DSL-convexity-short}, i.e., $c \in \left[\frac{n-1}2\pi, \frac{n+1}2\pi \right)$, convex in time?
\end{problem}

\subsection{Minimum Principle}

An important application of the previously known partial convexity of solutions to the DSL in the top branch is the minimum principle for Lagrangian graphs in that branch, established by Darvas--Rubinstein:

\begin{theorem}
{\rm \cite[Theorem~3.1]{Dar-Rub2019AMP}}
    Fix a domain $D\subset \RR^n$ and let $\dcal:= (0,1) \times D$. Let $u\in \fcal_c(\dcal)$ for branch $c\in \left[\frac{n}{2}\pi, \frac{n+1}2\pi \right)$. Then
    \[v(x):= \inf_{t\in (0,1)} u(t,x) \in F_{c-\frac{\pi}{2}}(D).\]
\end{theorem}

Does this extend to the second branch? More precisely:

\begin{problem}
    \label{problem:minimum-principle-DSL-convexity-short}
    Does a subsolution to the DSL \eqref{eq:the-DSL-DSL-convexity-short} satisfy a minimum principle in the second branch $c \in \left[\frac{n-1}2\pi, \frac{n}{2}\pi \right)$?
\end{problem}




\section{Results}

Problems~\ref{problem:joint-convexity-DSL-convexity-short}, \ref{problem:joint-convexity-2nd-branch-DSL-convexity-short} and \ref{problem:minimum-principle-DSL-convexity-short} are answered in the affirmative in this article. 

\subsection{Convexity of subsolutions to DSL}
\label{subsection:Results-DSL-convexity-short}

\begin{theorem}
\label{theorem:affine-slices-of-DSL-are-SLag-DSL-convexity-short}
Fix a bounded domain $D\subset \RR^n$ and let $\dcal := (0,1) \times D \subset \RR^{n+1}$. Fix $c \in \left[\frac{n-1}2\pi, \frac{n+1}2\pi \right)$, and consider a viscosity subsolution $u\in \fcal_c(\dcal)$. Then the following hold.
\begin{enumerate}[label={(\roman*)}]
    \item For each $x_0\in D$, $t\mapsto u(t,x_0)$ is a convex function on $(0,1)$.
    \item Let $H\subset \RR\times \RR^n$ be a 2-dimensional affine plane that does not contain $\RR\times \{x_0\} $ for any $x_0 \in \RR^n$. Then $u|_{\dcal\cap H}$ is subharmonic.
    \item If $c \in \left[ \frac{n}{2}\pi, \frac{n+1}2\pi \right)$, then $u$ is jointly convex on $\dcal$.
\end{enumerate}
\end{theorem}

Problem~\ref{problem:joint-convexity-DSL-convexity-short} is settled by Theorem~\ref{theorem:affine-slices-of-DSL-are-SLag-DSL-convexity-short}(iii) while Problem~\ref{problem:joint-convexity-2nd-branch-DSL-convexity-short} by Theorem~\ref{theorem:affine-slices-of-DSL-are-SLag-DSL-convexity-short}(i). By Example~\ref{example:second-branch-DSL-convexity-short}, a subsolution in the second branch of the DSL need not be 2-convex, but by Theorem~\ref{theorem:affine-slices-of-DSL-are-SLag-DSL-convexity-short}(ii) it is subharmonic along 2-dimensional planes that avoid a pure time direction.



\subsection{Minimum Principle in the second branch}

Ross--Witt-Nystr\"om introduced a notion of $\star$-product of two Dirichlet subequations (Definition \ref{def:product-of-dirichlet-sets-DSL-convexity-short}) \cite[\S 3.1]{Ross-Witt-Nystrom-minimum-principle-MR4350213} using which they generalized the Darvas--Rubinstein minimum principle for Lagrangian graphs. It turns out the top two branches of the DSL can be endowed with $\star$-product structures.

\begin{proposition}
    \label{prop:DSL-top-two-branches-are-products-convexity-short-note}
    For any $c\in \left[\frac{n-1}2\pi, \frac{n+1}2\pi \right)$ (refer to Definitions \ref{def:spacetimelagrangianangle-DSL-convexity-short} and \ref{def:product-of-dirichlet-sets-DSL-convexity-short}),
    \[ \fcal_c = \pcal_1 \star F_{c-\frac{\pi}{2}}. \]
    
\end{proposition}

Problem~\ref{problem:minimum-principle-DSL-convexity-short} is resolved in Corollary~\ref{cor:minimum-principle-DSL-convexity-short} by invoking the $\star$-product structure and the Ross--Witt-Nystr\"om minimum principle for product subequations \cite[Main Theorem]{Ross-Witt-Nystrom-minimum-principle-MR4350213}. Note that the chosen convention for the $\star$-product from Definition~\ref{def:product-of-dirichlet-sets-DSL-convexity-short} differs from the original definition of Ross--Witt-Nystr\"om \cite[Definition~3.1]{Ross-Witt-Nystrom-minimum-principle-MR4350213} because it reverses the roles of the two arguments. This change is chosen to stay consistent with the Darvas--Rubinstein minimum principle for Lagrangian graphs \cite[Theorem~3.1]{Dar-Rub2019AMP} (see Remark~\ref{remark:product-convention-DSL-convexity-short}).

\begin{corollary}
\label{cor:minimum-principle-DSL-convexity-short}
    Fix a domain $D\subset \RR^n$ and let $\dcal := (0,1) \times D$. For $c \in \left[\frac{n-1}2\pi, \frac{n+1}2\pi \right)$, if $u \in \fcal_c(\dcal)$ then
    \[v(x) := \inf_{t\in (0,1) } u(t,x)\in F_{c-\frac{\pi}{2}}(D).\]
\end{corollary}

\subsection{Legendre duality}
\label{sec:Legendre-duality-DSL-convexity-short}

A nice application of the minimum principle for product subequations is a Legendre duality between the Dirichlet problem for the product subequation and a family of obstacle problems for its Legendre transform. Ross--Witt-Nystr\"om first observed such a correspondence between test curves and weak geodesic rays in pluripotential theory as a consequence of the Kiselman’s minimum principle \cite[Theorem~1.1]{RWN-corresp-MR3194078}. Darvas--Rubinstein proved a minimum principle for the top branch of the DSL and as a consequence proved a duality between the Dirichlet problem for the DSL and a family of rooftop obstacle problems for the SL \cite[\S{4}]{Dar-Rub2019AMP}. Ross--Witt-Nystr\"om introduced the $\star$-product subequation and generalized the minimum principle to the $\star$-product. They remarked that such a Legendre duality for general product subequations should follow from their minimum principle \cite[p. 28]{Ross-Witt-Nystrom-minimum-principle-MR4350213}. Indeed Darvas--Rubinstein's original argument for the Legendre duality transform for the top branch of the DSL \cite[{\S}4]{Dar-Rub2019AMP} generalizes easily to general product subequations by the Ross--Witt-Nystr\"om minimum principle \cite[Main Theorem]{Ross-Witt-Nystrom-minimum-principle-MR4350213}. 
As an application of \eqref{prop:DSL-top-two-branches-are-products-convexity-short-note}, the Legendre duality transform extends to the second branch of the DSL by Proposition~\ref{prop:DSL-top-two-branches-are-products-convexity-short-note} to obtain a dual characterization of the Dirichlet problem for the DSL up to the second branch.

Fix a domain $D\subset \RR^n$ and a function $u: [0,1] \times D \rightarrow \RR$. Recall the following partial Legendre transform from \cite[(26)]{Dar-Rub2019AMP}
\[u^\star (\tau, x) := \inf_{t\in [0,1]} u(t,x) - t\tau \]
and its inverse transform \cite[(27)]{Dar-Rub2019AMP}
\[ v^\star (t,x) := \sup_{\tau \in \RR} v(\tau, x) +t \tau . \]
The partial Legendre transform defined here differs from the standard one
in convexity \cite[p.~104]{rockafeller-convex-analysis-MR1451876} by a sign. In particular, if $t\mapsto u(t,x)$ is convex, it holds that $u^{\star\star} =  u $. Also consider the rooftop envelope
\begin{equation}
    \label{eq:obstacle-envelope-DSL-convexity-short}
    P_{a}(h,f) := \sup \left\{ w\in F_a(D):
    \begin{gathered}
        w \le h \text{ on } D,\\ w\le f  \text{ on } \partial D
    \end{gathered}  \right\} \quad \forall\ h,f \in C^0(D),\ a\in \left(-\frac{n}{2} \pi, \frac{n}{2} \pi\right).
\end{equation}

\begin{corollary}
\label{cor:rooftop-obstacle-DSL-convexity-short}
    Fix a bounded domain $D\subset \RR^n$, let $\dcal := (0,1)\times D$, and fix a boundary condition $g\in C^0(\partial \dcal)$. Fix a phase angle in the top two branches $c \in [\frac{n-1}{2}\pi, \frac{n+1}{2} \pi)$. Assume that $\partial D$ is strictly $\vec{\tilde{F}}_{c-\pi/2}$-convex \cite[{\S}5]{HarLaw-dirichletduality}. Fix a function $u\in C^0(\dcal)$. Then $u$ is a viscosity solution to the Dirichlet problem
    \begin{equation}
    \label{eq:Dirichlet-problem-DSL-convexity-short}
        \begin{aligned}
            \widetilde{\Theta}(D^2u) = c& \qquad \text{on } \dcal,\\
            u = g& \qquad \text{on } \partial \dcal,
        \end{aligned}
    \end{equation}
    if and only if its partial Legendre transform $u^\star(\tau,\cdot)$ is a solution to a rooftop obstacle problem
    \begin{equation}
    \label{eq:DP-formula-1-DSL-convexity-short}
        u^\star(\tau,x) = P_{c-\frac{\pi}{2}}\left(\min\{g(0,\cdot),g(1,\cdot)-\tau\}, \inf_{t\in [0,1]} g(t,\cdot) - t\tau \right).
    \end{equation}
\end{corollary}

Since $t\mapsto u(t,x)$ is convex by Theorem~\ref{theorem:affine-slices-of-DSL-are-SLag-DSL-convexity-short}(i), \eqref{eq:DP-formula-1-DSL-convexity-short} yields a formula for viscosity solutions to the DSL in the top two branches
    \begin{equation}
        \label{eq:DP-formula-2-DSL-convexity-short}
        u(t,x) = \left( P_{c-\frac{\pi}{2}}\left(\min\{g(0,\cdot),g(1,\cdot)-\tau\}, \inf_{t\in [0,1]} \left[g(t,\cdot) - t\tau \right] \right) \right)^\star.
    \end{equation}
    
Corollary~\ref{cor:rooftop-obstacle-DSL-convexity-short} was established in \cite[Theorem~4.4]{Dar-Rub2019AMP} for the top branch $c\in [\frac{n}{2} \pi, \frac{n+1}{2} \pi)$ as a corollary of the minimum principle for Lagrangian graphs \cite[Theorem~3.1]{Dar-Rub2019AMP}. By Corollary~\ref{cor:minimum-principle-DSL-convexity-short}, the same argument also extends to the second branch $c\in [\frac{n-1}{2} \pi, \frac{n}{2} \pi)$. So a proof of Corollary~\ref{cor:rooftop-obstacle-DSL-convexity-short} is omitted and the reader is referred to \cite[{\S}4]{Dar-Rub2019AMP} for more details.

\subsection*{Organization}
Section~\ref{sec:preliminaries} recalls the definition of subequations and subsolutions, and recalls the subequations for the SL and DSL. Theorem~\ref{theorem:affine-slices-of-DSL-are-SLag-DSL-convexity-short} is proved in Section~\ref{sec:convexity-of-subsolutions-inthe-top-two-branches}. The definition of $\star$-products of subequations is recalled in Section~\ref{sec:Minimum-principle} followed by proofs of Proposition~\ref{prop:DSL-top-two-branches-are-products-convexity-short-note} and Corollary~\ref{cor:minimum-principle-DSL-convexity-short}. 

\subsection*{Acknowledgments}

Thanks to T. Darvas
for the second proof of Theorem \ref{theorem:affine-slices-of-DSL-are-SLag-DSL-convexity-short}(iii).
Research supported in part by NSF grants DMS-1906370, 2204347, 2506872, BSF grants 2016173, 2020329, and a Simons Fellowship in Mathematics.

\section{Preliminaries}
\label{sec:preliminaries}

This section recalls the definition of Dirichlet subequations, and subequations for SL \eqref{eq:the-SL-DSL-convexity-short} and DSL \eqref{eq:the-DSL-DSL-convexity-short} \cite{HarLaw-dirichletduality, Rub-Sol-DSL-MR3620699}. General facts about subequations are recalled in \S~\ref{subsection:dirichlet-sets-and-subequations-DSL-convexity-short} \cite{HarLaw-dirichletduality}. The lifted Lagrangian angle $\widetilde{\theta}$, lifted space-time Lagrangian angle $\widetilde{\Theta}$, and  subequations for SL, DSL are also recalled in \S\ref{subsection:subequations-for-the-DSL-DSL-convexity-short} \cite{HarLaw-dirichletduality, Rub-Sol-DSL-MR3620699}. 
In \S\ref{subsection:product-dirichlet-set-DSL-convexity-short}, a family of affine maps $\ell_{t_0,V}$ is defined, and the pullback of a Hessian under these maps is computed. The affine maps are used to prove Theorem~\ref{theorem:affine-slices-of-DSL-are-SLag-DSL-convexity-short}(iii) and Proposition~\ref{prop:DSL-top-two-branches-are-products-convexity-short-note}.

\subsection{Dirichlet subequations}
\label{subsection:dirichlet-sets-and-subequations-DSL-convexity-short}

\begin{definition}
\label{def:dirichlet-set-misc-definitions-DSL-convexity-short}
    The set of $n\times n$ symmetric positive semi-definite matrices is denoted by $\pcal_n = \{A \ge 0\} \subset \Sym^2(\RR^n)$.
    A \textit{Dirichlet subequation} $F \subset \Sym^2(\RR^n)$ is a non-empty proper closed subset satisfying 
    \[F+\pcal_n \subset F.\] 
    Basic examples of Dirichlet subequations are $\pcal_n$, and matrices of non-negative trace \[\tcal_n := \{ \tr(A) \ge 0 \} \subset \Sym^2(\RR^n).\]
\end{definition}


\begin{definition}
\label{def:subsolution-DSL-convexity-short}
    Fix a domain $D\subset \RR^n$. An upper semi-continuous function $u:D \rightarrow \RR$ is a \textit{viscosity subsolution to a subequation $F$} (or an \textit{$F$-subsolution}) if for every $x_0 \in D$ and every $\varphi \in C^2(D)$ satisfying $\varphi(x) \le u(x)$ and $\varphi(x_0) = u(x_0)$, $D^2\varphi(x_0) \in F$. The set of $F$-subsolutions is denoted by $F(D)$.
\end{definition}

If $u$ were a $C^2$ function, then $u$ is an $F$-subsolution at $x_0$ if and only if $D^2u(x_0)  \in F$. An \textit{$F$-solution} would be a subsolution satisfying $D^2u(x_0) \in \partial F$. This is typically represented by saying $D^2u(x_0) \in F$ and $D^2u(x_0) \in \overline{F^c}$, where $F^c = \Sym^2(\RR^n) \setminus F$ denote the complement.

\begin{definition}
\label{def:F-solution-DSL-convexity-short}
    An $F$-subsolution $u:D \rightarrow \RR$ is called an \textit{$F$-solution} or \textit{$F$-subharmonic} if
    \[u\in F(D) \quad \& \quad -u \in (\overline{-F^c})(D).\]
    The set $\overline{-F^c} \subset \Sym^2(\RR^n)$ is also a Dirichlet subequation, is called the \textit{Dirichlet dual} of $F$, and is denoted by $\overline{-F^c} =: \widetilde{F}$.
\end{definition}



The principal fact about Dirichlet subequations required in this article is the Harvey--Lawson restriction theorem.

\begin{theorem}
    \label{thm:restriction-theorem-DSL-convexity-short}
    {\rm\cite[Theorem~4.2]{HarLaw-Restriction-theorem-MR3330548}} Fix integers $m\le n$, and let $i:D_1 \rightarrow D_2$ be a $C^2$ embedding of an open domain $D_1 \subset \RR^m$ into another open domain $D_2 \subset \RR^n$. Let $F_1 \subset \Sym^2(\RR^m)$ and $F_2 \subset \Sym^2(\RR^n)$ be two Dirichlet subequations. Then the following are equivalent.
    \begin{enumerate}[label={(\roman*)}]
        \item For any $u\in C^2(D_2)$ and any $x\in D_1$ with $y:=i(x)$,
        \[D^2u(y) \in F_2 \quad \implies \quad D^2(u\circ i)(x) \in F_1.\]
        \item If $u\in F_2(D_2)$, then $i^*u := u\circ i \in  F_1(D_1)$.
    \end{enumerate}
\end{theorem}

\bremark
\label{remark:pcal-tcal-dirichlet-sets-DSL-convexity-short}
Since convexity and subharmonicity are associated to Dirichlet subequations $\pcal_n$ and $\tcal_n$, the restriction theorem also applies in these cases. Practically, this will
be used by restricting to affine subspaces of $\RR^{1+n}$.
\eremark

\subsection{Subequations for the SL and DSL}
\label{subsection:subequations-for-the-DSL-DSL-convexity-short}

This section recalls the Dirichlet subequations of interest in this article---the subequations for the SL in \eqref{eq:subeq-for-SL-DSL-convexity-short} and the DSL in \eqref{eq:subeq-for-DSL-DSL-convexity-short}.

\begin{definition}
    \label{def:SL-dirichlet-set-DSL-convexity-short}
    The \textit{lifted Lagrangian angle} $\widetilde{\theta}: \Sym^2(\RR^n) \rightarrow \RR$ is the function
    \begin{equation}
        \label{eq:lifted-lag-angle-DSL-convexity-short}
        \widetilde{\theta}(A) := \sum_{i=1}^n \tan^{-1}\lambda_i(A) \qquad \forall\ A\in \Sym^2(\RR^n),
    \end{equation}
    where $\lambda_1(A), \dots, \lambda_n(A)$ denote the eigenvalues of $A$, and where we choose the odd branch $\tan^{-1}:\RR \rightarrow \left(-\frac{\pi}{2}, \frac{\pi}{2} \right)$. The subequations for the SL \eqref{eq:the-SL-DSL-convexity-short} are
    \begin{equation}
        \label{eq:subeq-for-SL-DSL-convexity-short}
        F_c := \left\{A\in \Sym^2(\RR^n): \sum_{i=1}^n \tan^{-1} \lambda_i(A) \ge c \right\}, \qquad c\in \left( -\frac{n}{2} {\pi}, \frac{n}{2} {\pi} \right).
    \end{equation}
    The set of upper semi-continuous subsolutions is denoted by $F_c(D)$. If $u\in C^2(D)$, then $u$ is a subsolution to the SL \eqref{eq:the-SL-DSL-convexity-short} in branch $c$ if and only if $\widetilde{\theta}(D^2u(x)) \ge c$ for all $x\in D$.
\end{definition}

The Dirichlet subequations for the DSL \eqref{eq:the-DSL-DSL-convexity-short} are defined in terms of the lifted space-time Lagrangian angle $\widetilde{\Theta}$, which, unlike the lifted Lagrangian angle $\widetilde{\theta}$ is not continuous but merely upper semi-continuous. This is one of the main reasons the analysis of the DSL is much more subtle than that of the SL. Consider the following block representation for a matrix $A\in \Sym^2(\RR^{n+1})$,
    \begin{equation}
    \label{eq:blockrepresentationofDSLmatrix-DSL-convexity-short}
        A = \begin{pmatrix}
            a_{00} & \vec{a}^T \\
            \vec{a}& A^+
        \end{pmatrix}, \qquad a_{00} \in \RR,\, \vec{a} \in \RR^n,\, A^+ \in \Sym^2(\RR^n) .
    \end{equation}

\begin{definition}
\label{def:spacetimelagrangianangle-DSL-convexity-short}
Consider the diagonal matrix \[I_n:= \diag(0,1,\dots, 1) \in \Sym^2(\RR^{n+1})\]
and define a class of singular matrices by \[\scal := \{\diag(0,A^+): A^+ \in \Sym^2(\RR^n)\} \subset \Sym^2(\RR^{n+1}).\]
The \textit{lifted space-time Lagrangian angle} $\widetilde{\Theta}: \Sym^2(\RR^{n+1}) \rightarrow \RR$ is defined as
\begin{equation}
    \label{eq:lifted-spacetime-lag-angle-DSL-convexity-short}
    \widetilde{\Theta}(A) := \begin{dcases}
    \sum_{i=1}^{n+1} \arg\left(\lambda_i(I_n + \sqrt{-1}A)\right) \qquad & A \in \Sym^2(\RR^{n+1}) \setminus \scal , \\
    \widetilde{\theta}(A^+) + \frac{\pi}{2} \qquad & A\in \scal,
\end{dcases}
\end{equation}
where $\arg:\CC \setminus (-\infty, 0] \rightarrow \left(-\pi, \pi \right)$ is the counter-clockwise angle with the positive $x$-axis. The space-time Lagrangian angle $\widetilde{\Theta}$ is well defined by Lemma~\ref{lemma:Rub-Sol-Theta-well-def-DSL-convexity-short} since the eigenvalues of $A\in \Sym^2(\RR^{n+1}) \setminus \scal$ lie in $\CC \setminus (-\infty, 0]$. Furthermore, $\widetilde{\Theta}$ is the smallest upper semi-continuous function extending $\widetilde{\Theta}|_{ \Sym^2(\RR^{n+1}) \setminus \scal }$ by Lemma~\ref{lemma:Rub-Sol-continuous-DSL-convexity-short}.
\end{definition}


\begin{lemma}
\label{lemma:Rub-Sol-Theta-well-def-DSL-convexity-short}
    {\rm\cite[Lemma~3.3, 3.4]{Rub-Sol-DSL-MR3620699}} Let $A\in \Sym^2(\RR^{n+1})$ and $B = I_n + \sqrt{-1}A \in \Sym^2(\CC^{n+1})$ where $I_n = \diag(0,1\dots,1) \in \Sym^2(\RR^{n+1})$. Then the eigenvalues $\lambda_1,\dots, \lambda_n $ of $B$ satisfy \[\Re(\lambda_i) \ge 0 \qquad  i=1,\dots, n.\]
    Using notation \eqref{eq:blockrepresentationofDSLmatrix-DSL-convexity-short}, if $\vec{a} \neq 0$ then $\Re(\lambda_i) >0$ and if $a_{00} \neq 0$ then $\lambda_i \neq 0$ for all $i=1,\dots, n$.
\end{lemma}

\begin{lemma}
\label{lemma:Rub-Sol-continuous-DSL-convexity-short}
    {\rm\cite[Theorem~3.1]{Rub-Sol-DSL-MR3620699}} The space-time Lagrangian angle $\widetilde{\Theta}$, defined by \eqref{eq:lifted-spacetime-lag-angle-DSL-convexity-short}, is upper semi-continuous on $\Sym^2(\RR^{n+1})$ and continuous on $\Sym^2(\RR^{n+1}) \setminus \scal$.
\end{lemma}

\begin{definition}
    \label{def:DSL-dirichlet-set-DSL-convexity-short}
    The Dirichlet subequations for the DSL \eqref{eq:the-DSL-DSL-convexity-short} are the sets
    \begin{equation}
        \label{eq:subeq-for-DSL-DSL-convexity-short}
        \fcal_c:= \left\{ A \in \Sym^2(\RR^{n+1}) : \widetilde{\Theta}(A) \ge c \right\}, \qquad c\in \left( -\frac{n+1}2\pi, \frac{n+1}{2}\pi \right).
    \end{equation}
    $\fcal_c$ is closed since $\widetilde{\Theta}$ is upper semi-continuous and is a Dirichlet subequation by \cite[Theorem~5.1]{Rub-Sol-DSL-MR3620699}. These Dirichlet subequations define the various branches of the DSL \eqref{eq:the-DSL-DSL-convexity-short}.
\end{definition}

The lifted Lagrangian angle and space time Lagrangian angle are both functions of eigenvalues of their input and satisfy invariance under orthonormal change of coordinates.

\begin{claim}
\label{claim:DSL-partial-regularity-invariance-of-lag-angle-under-rotation}
    Let $U \in O(n)$ be an orthogonal matrix and $V := \diag(1, U) \in O(n+1)$. Then
    \begin{enumerate}[label={(\roman*)}]
        \item $\widetilde{\theta}(B) = \widetilde{\theta}(U^TBU)$ for all $B\in \Sym^2(\RR^n)$.
        \item $\widetilde{\Theta}(A) = \widetilde{\Theta}(V^TAV)$ for all $A\in \Sym^2(\RR^{n+1})$.
    \end{enumerate}
\end{claim}
\begin{proof}
    The eigenvalues of $B$ are the same as the eigenvalues of $U^TBU$, and the eigenvalues of $I_n + \sqrt{-1} A$ are equal to the eigenvalues of
    \[V^T(I_n +  \sqrt{-1}A)V = I_n + \sqrt{-1} V^TAV. \tag*{\qedhere}\]
\end{proof}

The following observations about the eigenvalues in the top two branches of the SL are well known \cite[p. 1357]{yuan-global-solutions-MR2199179}, \cite[Lemma~2.1]{collins-picard-wu-MR3659638}.

\begin{lemma}
\label{lemma:properties-of-eigenvalues-DSL-convexity-short}
    Let $A \in \Sym^2(\RR^n)$ and let $\lambda_1 \ge \dots \ge \lambda_n$ denote its eigenvalues.
    \begin{enumerate}[label={(\roman*)}]
        \item If $\widetilde{\theta}(A) \ge \frac{n-1}2\pi$, then $\lambda_1\ge \dots \ge \lambda_n \ge 0$. In particular $F_{c} \subset \pcal_n$ for $c \ge \frac{n-1}2\pi$.
        \item If $\widetilde{\theta}(A) \ge \frac{n-2}2\pi$, then $\lambda_1\ge \dots \ge \lambda_{n-1} \ge |\lambda_n| \ge 0$. In particular the sum of any two eigenvalues is non-negative.
    \end{enumerate}
\end{lemma}

\subsection{Hessian under affine maps}
\label{subsection:product-dirichlet-set-DSL-convexity-short}

This section introduces a family of affine maps $\ell_{t_0,V}$ in \eqref{eq:pull-back-under-ell-DSL-convexity-short-DSL} and computes the pullback of a Hessian under such a map (Claim~\ref{claim:derivative-and-pullback-DSL-convexity-short}). The affine maps are used in the proof of Theorem~\ref{theorem:affine-slices-of-DSL-are-SLag-DSL-convexity-short}. These affine maps also figure in \S\ref{sec:Minimum-principle}. 
Fix $m,n\in\NN$ and consider
\[\ell_{t_0, \Gamma}: \RR^n \rightarrow \RR^{m+n}, \qquad \ell_{t_0, \Gamma}(x):= (t_0 +\Gamma^T x, x) \qquad\forall\ t_0 \in \RR^m,\ \Gamma \in \RR^{n \times m}.\]
Note that the derivative of $\ell_{t_0, \Gamma}$ is independent of $t_0$,  and write \[\ell_{\Gamma} := \ell_{0, \Gamma}.\] 
A linear map $\ell_{t_0, \Gamma}$ induces a pullback on $(2,0)$-tensors and in particular Hessians of real valued functions on $\RR^{m+n}$. Identify $(2,0)$-tensors on $\RR^{m+n}$ with $(m+n) \times (m+n)$ matrices and consider the following block representation for any $A\in \Sym^2(\RR^{m+n})$
\begin{equation}
    \label{eq:n+m-matrix-decomp-DSL-convexity-short}
    A := \begin{pmatrix}
    B & C^T \\ 
    C & D
\end{pmatrix}, \qquad B\in \Sym^2(\RR^m),\, C \in \RR^{n\times m},\, D \in \Sym^2(\RR^n). 
\end{equation}
Let $I \in \Sym^2(\RR^n)$ denote the $n\times n$ identity matrix. Pullback under $\ell_{t_0, \Gamma}$ is given by
\begin{equation}
    \label{eq:pull-back-under-ell-DSL-convexity-short}
    \ell_{t_0, \Gamma}^* A = \ell_{0, \Gamma}^*A = \ell_{\Gamma}^*A = \begin{pmatrix}
        \Gamma & I
    \end{pmatrix} A \begin{pmatrix}
        \Gamma^T \\
        I
    \end{pmatrix} = \Gamma B\Gamma^T + \Gamma C^T + C \Gamma^T + D.
\end{equation}

The specific case of interest in this article is $m=1$, $F = \pcal_1 = [0,\infty) \subset \RR $, and $G = F_c$ a subequation for SL. In this case, in terms of the block representation \eqref{eq:blockrepresentationofDSLmatrix-DSL-convexity-short},
\begin{equation}
    \label{eq:pull-back-under-ell-DSL-convexity-short-DSL}
    \begin{gathered}
    \ell_{t_0, V}: \RR^n \rightarrow \RR^{n+1}, \qquad \ell_{t_0, V}(x):= (t_0 +V^Tx, x) \qquad \forall\ t_0 \in \RR,\ V \in \RR^{ n \times 1}, \\
    \ell_{t_0,V}^*A = \ell_{0,V}^*A = \ell_{V}^*A = \begin{pmatrix}
        V & I
    \end{pmatrix} A \begin{pmatrix}
        V^T\\
        I
    \end{pmatrix} = a_{00}VV^T + V \vec{a}^T + \vec{a} V^T + A^+.
    \end{gathered}
\end{equation}
According to the next computation, the action \eqref{eq:pull-back-under-ell-DSL-convexity-short-DSL} precisely corresponds to pullback when the matrices are Hessians.
\begin{claim}
\label{claim:derivative-and-pullback-DSL-convexity-short}
        Fix $u\in C^2(\RR \times \RR^n)$, and let $(t,x):= (t_0 + V^Tx, x) = \ell_{t_0,V}(x)$. Then
        \begin{equation}
            \label{eq:derivative-and-pullback-DSL-convexity-short}
             D^2 (u\circ \ell_{t_0,V})(x) = \ell_{{V}}^* D^2 u(t,x). 
        \end{equation}
    \end{claim}
    \begin{proof}
        Since $u\circ \ell_{t_0,V}(x) = u(t_0 + {V}^T x, x)$, in coordinates $t, x_1,\dots, x_n$ on $\RR\times \RR^n$,
    \begin{align*}
        \partial_{x_i} (u\circ \ell_{t_0,V})(x) =&\ \frac{\partial u}{\partial t}(t_0 + {V}^T x, x) \, {V}_i  + \partial_{x_i} u (t_0 + {V} \cdot x, x),\\
        \partial_{x_i}\partial_{x_j} (u\circ \ell_{t_0,V})(x) =&\ \left( {\partial^2_t u}\, {V}_i {V}_j + {\partial_{x_j} \partial_t u}\, {V}_i +  {\partial_t \partial_{x_i} u}\, {V}_j +  \partial_{x_i} \partial_{x_j} u \right)(t,x) 
    \end{align*}
    Consequently as a matrix equation,
    \begin{align*}    
        D^2 (u\circ \ell_{t_0,V})(x) =&\  \left( \partial_t^2 u {V} {V}^T + (\partial_tD_x u)  {V}^T + {V} (\partial_tD_x u)^T + D_x^2 u\right) (t,x) \\
                    =&\ \ell_{{V}}^* D^2 u(t,x) . \tag*{\qedhere}
    \end{align*}
    \end{proof}

\section{Convexity of subsolutions in the top two branches}
\label{sec:convexity-of-subsolutions-inthe-top-two-branches}

The goal of this section is to prove Theorem~\ref{theorem:affine-slices-of-DSL-are-SLag-DSL-convexity-short}. By Theorem~\ref{thm:restriction-theorem-DSL-convexity-short}, it suffices to prove Theorem~\ref{theorem:affine-slices-of-DSL-are-SLag-DSL-convexity-short} when $u \in C^2(\dcal)$. In this case Theorem \ref{theorem:affine-slices-of-DSL-are-SLag-DSL-convexity-short} reduces to a linear algebra statement about the Hessian $D^2u$, which is discussed in \S \ref{subsec:affine-slices-of-DSL} and \S\ref{subsec:time-covexity}. Theorem \ref{theorem:affine-slices-of-DSL-are-SLag-DSL-convexity-short} is proved in \S \ref{subsec:proof-of-theorem-1.1-(1)-(2)}.

\subsection{Affine slices of DSL subequations}
\label{subsec:affine-slices-of-DSL}

The proof of Theorem~\ref{theorem:affine-slices-of-DSL-are-SLag-DSL-convexity-short}(i) relies on an observation that affine sections of DSL subsolutions are SL subsolutions, i.e., if $u\in \fcal_c(\dcal)$ for $\dcal = (0,1) \times D$, then the affine slice $u\circ \ell_{t_0,V}$ is an $F_{c-\frac{\pi}{2}}$-subsolution on its domain of definition, where $\ell_{t_0,V}$ is as in \eqref{eq:pull-back-under-ell-DSL-convexity-short-DSL} (Corollary~\ref{cor:affine-slices-of-DSL-are-SL-DSL-convexity-short}). A nice consequence of this observation is that subsolutions to the second branch of the DSL are subharmonic along $2$-dimensional planes not containing time-like lines of the form $\RR\times \{x_0\}$ (Corollary~\ref{cor:2-convex-second-branch-DSL-convexity-short}). This is optimal by Example~\ref{example:second-branch-DSL-convexity-short}. If $u\in C^2(\dcal) \cap \fcal_c(\dcal)$ then Corollary~\ref{cor:affine-slices-of-DSL-are-SL-DSL-convexity-short} reduces to the following linear algebra observation about the Hessian $D^2u$ (compare with Claim~\ref{claim:derivative-and-pullback-DSL-convexity-short}). 
 
\begin{proposition}
\label{prop:affine-slices-of-DSL-matrix-is-SLag-DSL-convexity-short}
    Fix a matrix $A \in \Sym^2(\RR^{n+1})$. Recall the lifted Lagrangian angle $\widetilde{\theta}$ \eqref{eq:lifted-lag-angle-DSL-convexity-short}, lifted space-time Lagrangian angle $\widetilde{\Theta}$ \eqref{eq:lifted-spacetime-lag-angle-DSL-convexity-short}, and the pullback $\ell_{V}^*$ \eqref{eq:pull-back-under-ell-DSL-convexity-short-DSL}. Then for any $c\in  \left(-\frac{n+1}{2}\pi, \frac{n+1}2\pi \right)$ and ${V} \in \RR^{n\times 1}$,
    \begin{equation}
        \label{eq:affine-slices-result-DSL-convexity-short}
        \widetilde{\Theta}(A) \ge c \quad \implies \quad \widetilde{\theta} (\ell_{{V}}^*A) \ge c  - \frac{\pi}{2}.
    \end{equation}
    In other words, $A \in \fcal_c$ implies $\ell_V^*A \in F_{c-\frac{\pi}{2}}$ for all $V\in \RR^{n\times 1}$.
\end{proposition}

An ingredient in the proof of Proposition~\ref{prop:affine-slices-of-DSL-matrix-is-SLag-DSL-convexity-short} and later in the proof of Lemma~\ref{lemma:time-convexity-DSL-convexity-short} is a useful formula:

\begin{claim}
\label{claim:lag-ang-ST-lag-ang-formula-DSL-convexity-short}
     {\rm\cite[Lemma~3.6, Lemma~3.7]{Rub-Sol-DSL-MR3620699}} Let $A \in \Sym^2(\RR^{n+1})$ and use the notation \eqref{eq:blockrepresentationofDSLmatrix-DSL-convexity-short}. Assume that $A\neq \diag(0,A^+)$. Then
    \begin{equation}
        \label{eq:lag-ang-ST-lag-ang-formula-DSL-convexity-short}
        \widetilde{\Theta}(A) - \widetilde{\theta}(A^+) = \arg \left( \sqrt{-1} a_{00} + \vec{a}^T (I + \sqrt{-1}A^+)^{-1} \vec{a}  \right) \in \left( -\frac{\pi}{2}, \frac{\pi}{2} \right),
    \end{equation}
    where $\arg: \CC \setminus (-\infty, 0] \rightarrow \left(-\pi, \pi \right)$ is the counter-clockwise angle to the positive x-axis.
\end{claim}

\begin{proof}[Proof of Proposition~\ref{prop:affine-slices-of-DSL-matrix-is-SLag-DSL-convexity-short}]
First observe that if $A= \diag(0,A^+)$, then $\ell_{V}^* A = A^+$ and \eqref{eq:affine-slices-result-DSL-convexity-short} follows by definition of $\widetilde{\Theta}$ \eqref{eq:lifted-spacetime-lag-angle-DSL-convexity-short}. So assume that $A\notin \scal$, and note that \eqref{eq:lag-ang-ST-lag-ang-formula-DSL-convexity-short} applies. There are two cases to consider.

\medskip
\noindent
\textit{Case (i):} ${V} = 0$. 
Since $\ell_{0}^* A = A^+$, by \eqref{eq:lag-ang-ST-lag-ang-formula-DSL-convexity-short},
    \[\widetilde{\Theta}(A) - \widetilde{\theta}(A^+) = \arg \left( \sqrt{-1} a_{00} + \vec{a}^T (I + \sqrt{-1}A^+)^{-1} \vec{a}  \right) \in \left( - \frac{\pi}{2}, \frac{\pi}{2} \right).\]
\medskip
\noindent
    \textit{Case (ii):} ${V} \neq 0$. The idea is to
    consider $A$ as a $(2,0)$-tensor and pull it back  under the map $(t,x)\mapsto \ell_{t,V}(x)$. It turns out that thanks to
    Lemma~\ref{lemma:spectrum-of-AGamma-convexity-note} below, Case (i) can be applied to this pullback. For $A\in \Sym^2(\RR^{n+1})$ and ${V} \in \RR^{n\times 1}$ let
    \begin{equation}
        \label{eq:eq6-DSL-convexity-short}
        A_{V} := \begin{pmatrix}
        1 & 0 \\
        {V} & I
    \end{pmatrix} A \begin{pmatrix}
        1 & {V}^T \\
        0 & I
    \end{pmatrix}\in \Sym^2(\RR^{n+1}).
    \end{equation}
By Lemma~\ref{lemma:pull-back-and-change-of-basis-DSL-convexity-short} below, $(A_{{V}})^+ = \ell_{{V}}^*A$.
    So applying Case (i) to $A_{V}$ (think $A_V=(A_V)_0$),
    $$
\widetilde{\Theta}(A_V) - \widetilde{\theta}(\ell_{{V}}^*A)
=
\widetilde{\Theta}(A_V) - \widetilde{\theta}(A_V^+) \in \left( - \frac{\pi}{2}, \frac{\pi}{2} \right).
    $$
Thus, by Lemma~\ref{lemma:spectrum-of-AGamma-convexity-note} below,
$$
\widetilde{\Theta}(A) - \widetilde{\theta}(\ell_{{V}}^*A)
\in \left( - \frac{\pi}{2}, \frac{\pi}{2} \right),
    $$
and the assumption $A \in \fcal_c$ yields
    \[\widetilde{\theta}(\ell_{V}^* A) = \widetilde{\theta}(A_{V}^+)  \ge  \widetilde{\Theta}(A_{{V}}) - \frac{\pi}{2} \ge c-\frac{\pi}{2}.\]
    This completes the proof of Case (ii) and hence of 
    Proposition~\ref{prop:affine-slices-of-DSL-matrix-is-SLag-DSL-convexity-short}.
\end{proof}

Next, we establish the two lemmas used in the proof of 
Proposition~\ref{prop:affine-slices-of-DSL-matrix-is-SLag-DSL-convexity-short}.
\begin{lemma}
\label{lemma:pull-back-and-change-of-basis-DSL-convexity-short}
    Fix $A\in \Sym^2(\RR^{n+1})$ and $V\in \RR^{n\times1}$ and let $A_V$ be as in \eqref{eq:eq6-DSL-convexity-short}. Then
    \begin{equation}
    \label{eq:eq7-DSL-convexity-short}
        A_V = \begin{pmatrix}
        a_{00} & a_{00}{V}^T + \vec{a}^T \\
        a_{00}{V} + \vec{a} &   a_{00}{V} {V}^T + \vec{a} {V}^T + {V} \vec{a}^T + A^+
    \end{pmatrix}.
    \end{equation}
    In particular using the block representation \eqref{eq:blockrepresentationofDSLmatrix-DSL-convexity-short} and the pullback \eqref{eq:pull-back-under-ell-DSL-convexity-short-DSL}, $(A_V)^+ = \ell_{V}^*A$.
\end{lemma}
\begin{proof}
    Under the block representation \eqref{eq:blockrepresentationofDSLmatrix-DSL-convexity-short}, by \eqref{eq:pull-back-under-ell-DSL-convexity-short-DSL}
    \begin{align*}
       A_{{V}} :=  \begin{pmatrix}
        1 & 0 \\
        {V} & I
    \end{pmatrix} A \begin{pmatrix}
        1 & {V}^T \\
        0 & I
    \end{pmatrix} =&\ \begin{pmatrix}
        1 & 0 \\
        {V} & I
    \end{pmatrix} \begin{pmatrix}
        a_{00} & \vec{a}^T \\
        \vec{a} & A^+
    \end{pmatrix} \begin{pmatrix}
        1 & {V}^T \\
        0 & I
    \end{pmatrix} \\
    =&\ \begin{pmatrix}
        a_{00} & a_{00}{V}^T + \vec{a}^T \\
        a_{00}{V} + \vec{a} &   a_{00}{V} {V}^T + \vec{a} {V}^T + {V} \vec{a}^T + A^+
    \end{pmatrix} \tag*{\qedhere}.
    \end{align*}
\end{proof}

    \begin{lemma}
    \label{lemma:spectrum-of-AGamma-convexity-note}
        For $A\in \Sym^2(\RR^{n+1})$ and ${V} \in \RR^{n\times 1}$, $\widetilde{\Theta}(A) = \widetilde{\Theta}(A_{{V}})$ (recall (\ref{eq:eq6-DSL-convexity-short})).
    \end{lemma}
    \begin{proof}
    First suppose that $A\in \scal$, i.e., $A = \diag(0,A^+)$. By \eqref{eq:pull-back-under-ell-DSL-convexity-short-DSL}, $\ell_{V}^*A = A^+$ and $A_V = A$ and so $\widetilde{\Theta}(A_V)= \widetilde{\Theta}(A)$. So assume without loss of generality that $A\notin \scal$.
    \begin{claim}
    \label{claim:image-of-AV-DSL-convexity-short}
        If $A\notin \scal$, then $A_V \notin \scal$ for all $V\in \RR^{n\times 1}$.
    \end{claim}
    \begin{proof}
    If $a_{00} \neq 0$ then by \eqref{eq:eq6-DSL-convexity-short} $A_V \notin \scal$ for all $V\in \RR^{n+1}$, since its $(1,1)$ entry is always $a_{00}$ and hence non-zero. Suppose $a_{00} = 0$. Then by assumption $\vec{a} \neq 0$, and
    \[A_V = \begin{pmatrix}
        0 & \vec{a}^T \\
        \vec{a} & \vec{a}V^T + V \vec{a}^T + A^+,
    \end{pmatrix} \notin \scal. \tag*{\qedhere}\]
    \end{proof}
     $A\notin \scal$ by assumption and by Claim~\ref{claim:image-of-AV-DSL-convexity-short} $A_V \notin \scal$ for all $V\in \RR^{n\times 1}$. Then by \eqref{eq:lifted-spacetime-lag-angle-DSL-convexity-short}
     \begin{equation}
         \label{eq:Theta-formula-special-case-DSL-convexity-short}
         \widetilde{\Theta}(A_V) \equiv_{2\pi} \arg \det(I_n + \sqrt{-1} A_V) \qquad \forall\ V\in \RR^{n\times1}.
     \end{equation}
    Recall $I_n = \diag(0,1,\dots, 1) \in \Sym^2(\RR^{n+1})$, and let $I$ denote the $n\times n$ identity matrix.
    \[(I_n)_{{V}} := \begin{pmatrix}
        1 & 0 \\
        {V} & I
    \end{pmatrix} I_n \begin{pmatrix}
        1 & {V}^T \\
        0 & I
    \end{pmatrix} = \begin{pmatrix}
        1 & 0 \\
        {V} & I
    \end{pmatrix} \begin{pmatrix}
        0 & 0 \\
        0 & I
    \end{pmatrix} \begin{pmatrix}
        1 & {V}^T \\
        0 & I
    \end{pmatrix} = \begin{pmatrix}
        0 & 0 \\
        0 & I
    \end{pmatrix} = I_n.\]
    Consequently,
    \[\left(I_n + \sqrt{-1} A \right)_V = \begin{pmatrix}
        1 & 0 \\
        {V}^T & I
    \end{pmatrix} \left(I_n + \sqrt{-1} A \right) \begin{pmatrix}
        1 & {V} \\
        0 & I
    \end{pmatrix} = I_n + \sqrt{-1} A_{{V}}. \]
    The determinant of a lower triangular matrix with $1$'s along the diagonal is $1$, so both sides of the previous equation have equal determinants. Thus
    \[
        \qquad \arg \det(I_n + \sqrt{-1} A_{V}) \equiv_{2\pi} \arg \det(I_n + \sqrt{-1} A).
    \]
    Conclude by \eqref{eq:Theta-formula-special-case-DSL-convexity-short},
    \[\begin{aligned}
        \widetilde{\Theta}(A) \equiv_{2\pi} \arg \det(I_n + \sqrt{-1} A)\equiv_{2\pi}&\ \arg \det(I_n + \sqrt{-1} A_{V}) \equiv_{2\pi} \widetilde{\Theta}(A_{V}),
    \end{aligned}\]
    \[\widetilde{\Theta}(A) -  \widetilde{\Theta}(A_{V} ) \in 2\pi \ZZ \qquad \forall\ V\in \RR^{n\times 1}.\]
    But $\widetilde{\Theta}$ is continuous on the set $\{A_V: V\in \RR^{n+1}\} \subset \Sym^2(\RR^{n+1})$ by Claim~\ref{claim:image-of-AV-DSL-convexity-short} and Lemma~\ref{lemma:Rub-Sol-continuous-DSL-convexity-short}. So ${V} \mapsto \widetilde{\Theta}(A_{{V}}) - \widetilde{\Theta}(A)$ is a continuous real valued map taking values in $2\pi \ZZ$ which is zero when $v=0$. Conclude that
    \begin{equation}
        \widetilde{\Theta}(A) =  \widetilde{\Theta}(A_{V} ) \qquad \forall\ {V} \in \RR^{n\times 1}. \tag*{\qedhere}
    \end{equation}
    \end{proof}


\begin{corollary}
    \label{cor:affine-slices-of-DSL-are-SL-DSL-convexity-short}
    Fix a domain $\dcal\subset \RR^{n+1}$ and a branch $|c| \le \frac{n+1}2\pi$. Pick any $t_0\in \RR$ and $V\in \RR^{n+1}$, and consider the pullback domain $D_{t_0,V} := \ell_{t_0, V}^{-1}(\dcal)\cap D$ where $\ell_{t_0,V}$ is as in \eqref{eq:pull-back-under-ell-DSL-convexity-short-DSL}. If $u\in \fcal_c(\dcal)$, then $u\circ\ell_{t_0,V} \in F_{c-\frac{\pi}{2}}(D_{t_0,V})$.
\end{corollary}
\begin{proof}
    First suppose that $u\in C^2(\dcal)\cap \fcal_c(\dcal)$. Consider a point $x\in D$ such that $(t,x) := \ell_{t_0,V}(x) \in \dcal = (1,0)\times D$. Then $D^2(u\circ  \ell_{t_0,V})(x) = \ell_{V}^*D^2u(t,x)$ by \eqref{eq:derivative-and-pullback-DSL-convexity-short}. Since $D^2u(t,x) \in \fcal_{c}$, $D^2(u\circ \ell_{t_0,V})(x) = \ell_{V}^*D^2u(t,x) \in F_{c-\frac{\pi}{2}}$ by Proposition~\ref{prop:affine-slices-of-DSL-matrix-is-SLag-DSL-convexity-short}.
    
    Corollary~\ref{cor:affine-slices-of-DSL-are-SL-DSL-convexity-short} then follows for general $u\in \fcal_c(\dcal)$ by Theorem~\ref{thm:restriction-theorem-DSL-convexity-short}.
\end{proof}

\begin{remark}
    By Definition~\ref{def:subsolution-DSL-convexity-short}, an upper semi-continuous function being an $F$-subsolution is a local property. So $u\circ \ell_{t_0,V}$ is $F_{c-\frac{\pi}{2}}$-subharmonic at $x\in D$ when ever $u$ is $\fcal_c$-subharmonic in a neighborhood $U$ of $\ell_{t_0,V}(x) \in \dcal$, by applying Theorem~\ref{thm:restriction-theorem-DSL-convexity-short} to the embedding $\ell_{t_0,V}$ restricted to a neighborhood of $x$ landing inside $U$.
\end{remark}

\begin{remark}
\label{remark:boundary-problems-DSL-convexity-short}
    If $\dcal := (0,1) \times D$ for $D\subset \RR^n$ as in Theorem~\ref{theorem:affine-slices-of-DSL-are-SLag-DSL-convexity-short}, then 
    \[ D_{t_0,0} := \ell_{t_0,0}^{-1}(\dcal) = \begin{cases}
         D \qquad &t_0\in (0,1),\\
        \varnothing \qquad &t_0 \in \RR \setminus (0,1).
    \end{cases}\]
    In particular by Corollary~\ref{cor:affine-slices-of-DSL-are-SL-DSL-convexity-short} applies to the special case $V=0$ to yield
    \begin{equation}
        \label{eq:cant-extend-to-bdry-DSL-convexity-short}
        u(t_0,\cdot) \in F_{c-\frac{\pi}{2}}(D) \qquad \forall\ t_0 \in (0,1),\ u\in \fcal_c(\dcal).
    \end{equation}
    The property \eqref{eq:cant-extend-to-bdry-DSL-convexity-short} was first observed in \cite[Lemma~5.1]{Rub-Sol-DSL-MR3620699}. Interestingly \eqref{eq:cant-extend-to-bdry-DSL-convexity-short} cannot be extended to $t\in \{0,1\}$ even if $u$ is continuous up to the boundary of $\dcal$ \cite[Remark~9.3]{Rub-Sol-DSL-MR3620699}.
\end{remark}

\begin{corollary}
\label{cor:2-convex-second-branch-DSL-convexity-short}
    Fix a domain $\dcal\subset \RR\times \RR^{n}$. Fix a branch $c \in \left[\frac{n-1}2\pi, \frac{n+1}2\pi \right)$. Let $H\subset \RR\times \RR^n$ be a 2-dim affine plane that does not contain lines of the form $\RR\times \{x_0\}$ for any $x_0\in \RR^n$. Then $u|_{\dcal\cap H}$ is subharmonic.
\end{corollary}
\begin{proof}
    By Claim~\ref{claim:2-dim-plane-DSL-convexity-short}, $H$ lies inside the image of $ \ell_{t_0,V}$ for some $V\in \RR^{n\times 1}$ (see \eqref{eq:pull-back-under-ell-DSL-convexity-short-DSL}). By Corollary~\ref{cor:affine-slices-of-DSL-are-SL-DSL-convexity-short}, $u\circ \ell_{t_0, V}$ is a $F_{c-{\pi}/{2}}$-subsolution. But by Lemma~\ref{lemma:properties-of-eigenvalues-DSL-convexity-short}, a $F_{c-{\pi}/{2}}$-subsolution which is also $C^2$ is 2-convex since $c-\frac{\pi}{2} \ge \frac{n-2}2\pi$, and in particular subharmonic on the pullback of $H$ under $\ell_{t_0,V}$. By Remark~\ref{remark:pcal-tcal-dirichlet-sets-DSL-convexity-short}, it follows that any $F_{c-\frac{\pi}{2}}$-subsolution is subharmonic along $\ell_{t_0,V}^{-1}H$. Since $\ell_{t_0,V}$ is a linear map, it follows that $u|_{H\cap \dcal}$ is subharmonic.
\end{proof}

Claim~\ref{claim:2-dim-plane-DSL-convexity-short} is a general result about affine subspaces of $\RR^{n+1}$.

    \begin{claim}
    \label{claim:2-dim-plane-DSL-convexity-short}
        Let $H\subset \RR\times \RR^n$ be a 2-dim affine plane that contains no line of the form $\RR\times \{x_0\}$. Then $H \subset P_{t_0,V}:= \{(t_0 + \langle V,x\rangle , x): x\in \RR^n\}$ for some $t_0\in \RR$, $V\in \RR^n$.
    \end{claim}
    \begin{proof}
        Fix $(t',x') \in H$ and consider the subspace $H_0:= H-(t',x')$. By assumption, $(1,0) \notin H_0$. Suppose $H_0 = \operatorname{span}\{h_1,h_2\}$ where $h_1 = (t_1,x_1)$ and $h_2 = (t_2,x_2)$. Note that $x_1,x_2 \in \RR^n$ are linearly independent because $(1,0) \notin H_0$. Explicitly, if $x_1 = c x_2$, then $ct_2 - t_1 \neq 0$ since $H_0$ is 2-dimensional, and $(1,0) =  (ct_2 - t_1)^{-1} (h_2 - h_1) \in H_0$ which contradicts the assumption on $H$.
        
        Since $x_1,x_2 $ are linearly independent, there exists $V\in \RR^n$ solving the linear system
        \[\begin{pmatrix}x_1\cr x_2\end{pmatrix} V =
        \begin{pmatrix}t_1\cr t_2\end{pmatrix}. \]
        In other words $\langle V,x_1 \rangle = t_1$ and $\langle V, x_2 \rangle = t_2$. Set $t_0 := t' - \langle V,x'\rangle$. It turns out $H \subset P_{t_0,V}$ for this choice of $t_0$ and $V$. Consider an arbitrary element $(t,x)\in H$. Since $(t,x) - (t',x') \in H_0$, there exist $c_1,c_2\in \RR$ such that $(t,x) - (t',x') = c_1 (t_1,x_1) + c_2(t_2,x_2)$. Then
        \begin{align*}
            (t,x) =&\ \big(t' + \langle V, c_1x_1+c_2x_2 \rangle, x' + c_1x_1 + c_2x_2\big) \\
            =&\ \big(t_0 + \langle V, x' + c_1 x_1+ c_2x_2 \rangle, x' + c_1x_1 + c_2x_2\big) \in P_{t_0,V}.
        \end{align*}
        Since $(t,x)$ was arbitrary, $H \subset P_{t_0,V}$.
    \end{proof}

\subsection{Time convexity}
\label{subsec:time-covexity}

This section establishes convexity in time for subsolutions $u \in \fcal_c(\dcal)$ in the top two branches of the DSL, $c\ge \frac{n-1}2\pi$ (Corollary~\ref{cor:time-convexity-DSL-convexity-short}). As in {\S}\ref{subsec:affine-slices-of-DSL}, the primary step in the argument is a linear algebra observation (Lemma~\ref{lemma:time-convexity-DSL-convexity-short}) which establishes the result for $u\in C^2(\dcal)\cap \fcal_c(\dcal)$.

\begin{lemma}
\label{lemma:time-convexity-DSL-convexity-short}
     Suppose 
     $c\in \left[\frac{n-1}2\pi,\frac{n+1}2\pi \right)$.
     For $A \in \fcal_c\subset \Sym^2(\RR^{n+1})$,
     $a_{00} \ge 0$ (recall the notation (\ref{eq:blockrepresentationofDSLmatrix-DSL-convexity-short})). Furthermore if $\vec{a} \neq 0 $ then $a_{00} >0$.
\end{lemma}
\begin{proof}
    Let $U \in O(n)$ be an orthogonal matrix diagonalizing $A^+$, i.e., 
            \[U^T A^+ U = \diag( \lambda_1,\dots, \lambda_n).\]
        For $V := \diag(1, U) \in O(n+1)$ and using notation \eqref{eq:blockrepresentationofDSLmatrix-DSL-convexity-short},
        \begin{align}
            V^T A V =&\ \begin{pmatrix}
            1 & 0 \\
            0 & U^T
        \end{pmatrix} \begin{pmatrix}
        a_{00} & a^T\\
        a & A^+
    \end{pmatrix} \begin{pmatrix}
            1 & 0 \\
            0 & U
        \end{pmatrix} = \begin{pmatrix}
        a_{00} & a^TU\\
        U^Ta & U^TA^+U
    \end{pmatrix} \nonumber \\
        =&\ \begin{pmatrix}
    a_{00} & \vec{b}^T \\
    \vec{b} &\begin{matrix}\lambda_{1} & & \\ & \ddots & \\ & & \lambda_{n}\end{matrix}
    \end{pmatrix} \label{eq:time-convexity-notation-1-DSL-convexity-short}
        \end{align}
        where $\vec{b} = (b_1,\dots, b_n):= U^T \vec{a}$.  Note that $(V^TAV)^+ = U^TA^+U$ by \eqref{eq:time-convexity-notation-1-DSL-convexity-short}. Then by \eqref{eq:lag-ang-ST-lag-ang-formula-DSL-convexity-short}, Claim~\ref{claim:DSL-partial-regularity-invariance-of-lag-angle-under-rotation}, and the definition of $\widetilde{\theta}$ in \eqref{eq:lifted-lag-angle-DSL-convexity-short},
            \begin{align}
                \widetilde{\Theta}(A) - \widetilde{\theta}(A^+) =&\ \widetilde{\Theta}(V^TAV) - \widetilde{\theta}(U^TA^+U) \nonumber \\
                =&\ \arg \left( \sqrt{-1} a_{00} + \vec{b}^T (I + \sqrt{-1}U^TA^+U )^{-1} \vec{b} \right)  \nonumber \\
                =&\ \arg \left( \sqrt{-1} a_{00} + \vec{b}^T \diag(1+\sqrt{-1}\lambda_1,\dots, 1+\sqrt{-1}\lambda_n) ^{-1} \vec{b} \right)  \nonumber \\
                =&\ \arg \left( \sqrt{-1} a_{00} + \sum_{i=1}^n \frac{b_i^2 }{1+\sqrt{-1}\lambda_i} \right)  \nonumber \\
                =&\ \arg \left( \sqrt{-1} a_{00} + \sum_{i=1}^n \frac{b_i^2(1-\sqrt{-1}\lambda_i) }{1+\lambda_i^2} \right)  \nonumber \\
                =&\ \arg \left( \sum_{i=1}^n \frac{b_i^2 }{1+\lambda_i^2} + \sqrt{-1} a_{00} - \sqrt{-1} \sum_{i=1}^n \frac{b_i^2\lambda_i }{1+\lambda_i^2} \right). \label{eq:DSL-partial-regularity-second-branch-time-convexity-eq1.1-1}
            \end{align}
            \noindent \textbf{Case 1:} $\vec{a} = 0$. Suppose on the contrary that $a_{00} <0$. Then by \eqref{eq:DSL-partial-regularity-second-branch-time-convexity-eq1.1-1}
            \[\widetilde{\Theta}(A) = \widetilde{\theta}(A^+) + \arg(\sqrt{-1} a_{00}) = \widetilde{\theta}(A^+) - \frac{\pi}{2}  < \frac{n\pi}{2} - \frac{\pi}{2} = (n-1)\frac{\pi}{2},\]
            contradicting the assumption $A \in \fcal_c = \{\widetilde{\Theta} \ge c\}$ for $c \ge \frac{n-1}{2} \pi$. Thus $a_{00} \ge 0$ in this case.
            
            \noindent \textbf{Case 2:} $\vec{a} \neq 0$. Observe that $\vec{b} = U^T \vec{a} \neq 0$ since $U$ is orthogonal. Also assume on the contrary that $a_{00} \le 0$. By \eqref{eq:DSL-partial-regularity-second-branch-time-convexity-eq1.1-1}, the fact that $U^TA^+U$ is diagonal and by \eqref{eq:lifted-lag-angle-DSL-convexity-short},
            \begin{equation}
                \label{eq:BIrrDSL-DSL-partial-regularity-second-branch-time-convexity-eq1.1}
                 \widetilde{\Theta}(A) = \sum_{i=1}^n \tan^{-1} \lambda_i +  \tan^{-1} \left( \frac{a_{00} - \sum_{i=1}^n \frac{b_i^2\lambda_i }{1+\lambda_i^2}}{\sum_{i=1}^n \frac{b_i^2 }{1+\lambda_i^2}} \right) .
            \end{equation}
        The next claim gives an upper bound on \eqref{eq:BIrrDSL-DSL-partial-regularity-second-branch-time-convexity-eq1.1}.
            \begin{claim}
            \label{claim:DSL-partial-regularity-2-branch-t-conv-claim-1.1}
            Suppose $a_{00} \le 0$, $\vec{b} = (b_1,\dots, b_n) \in \RR^n$ such that $\vec{b} \neq 0$, and $\lambda_1\ge \dots \ge \lambda_n$. Then
                \[\frac{a_{00} - \sum_{i=1}^n \frac{b_i^2\lambda_i }{1+\lambda_i^2}}{\sum_{i=1}^n \frac{b_i^2 }{1+\lambda_i^2}} \le - \min_{1\le i\le n} \lambda_i = - \lambda_n. \]
            \end{claim}
            \begin{proof}
                Let $\mu_i := b_i^2/(1+\lambda_i^2)$ denote non-negative weights. Since a minimum is always smaller than the weighted average,
                \[ \frac{\sum_{i=1}^n \frac{b_i^2\lambda_i }{1+\lambda_i^2}}{\sum_{i=1}^n \frac{b_i^2 }{1+\lambda_i^2}} = \frac{\sum_{i=1}^n \lambda_i \mu_i}{\sum_{i=1}^n \mu_i}  \ge \min_{1\le i\le n} \lambda_i. \]
                Since $a_{00} \le 0$ by assumption,
                \begin{align*}
                    \frac{a_{00} - \sum_{i=1}^n {\mu_i \lambda_i }}{\sum_{i=1}^n \mu_i} =&\ \frac{a_{00}}{\sum_{i=1}^n \mu_i}  - \frac{ \sum_{i=1}^n \mu_i \lambda_i}{\sum_{i=1}^n \mu_i}  
                    \le\  0 - \min_{1\le i\le n} \lambda_i = - \lambda_n. \tag*{\qedhere}
                \end{align*}
            \end{proof}
            By Claim~\ref{claim:DSL-partial-regularity-2-branch-t-conv-claim-1.1} applied to \eqref{eq:BIrrDSL-DSL-partial-regularity-second-branch-time-convexity-eq1.1}, \[\begin{aligned}
                \widetilde{\Theta}(A) =&\ \sum_{i=1}^n \tan^{-1} \lambda_i + \tan^{-1} \left( \frac{a_{00} - \sum_{i=1}^n \frac{b_i^2\lambda_i }{1+\lambda_i^2}}{\sum_{i=1}^n \frac{b_i^2 }{1+\lambda_i^2}} \right)  \\
                \le &\ \sum_{i=1}^n \tan^{-1} \lambda_i + \tan^{-1} (-\lambda_n)  \\
                =&\ \sum_{i=1}^{n-1} \tan^{-1} \lambda_i 
                < \frac{n-1}{2}\pi.
            \end{aligned}\]
            This contradicts the assumption $\widetilde{\Theta}(A)  \ge \frac{n-1}{2}\pi$. Thus $a_{00} > 0$ in this case. The two cases together complete the proof of Lemma~\ref{lemma:time-convexity-DSL-convexity-short}.
\end{proof}

Corollary~\ref{cor:time-convexity-DSL-convexity-short-C2} follows directly from Lemma~\ref{lemma:time-convexity-DSL-convexity-short} applied to the Hessian matrix.

\begin{corollary}
    \label{cor:time-convexity-DSL-convexity-short-C2}
    Consider $\dcal := (0,1)\times D$ for some domain $D \subset \RR^{n}$ and let $c\in \left[\frac{n-1}2\pi, \frac{n+1}2\pi\right)$. If $u\in \fcal_c(\dcal) \cap C^2(\dcal)$ then $t\mapsto u(t,x_0)$ is convex for $t\in (0,1)$ for each $x_0 \in D$. Furthermore if $\partial_t D_x u(t_0,x_0) \neq 0 $ for some $(t_0, x_0) \in \dcal$ then $\partial_t^2 u (t_0,x_0) >0$, i.e., $u(\cdot, x_0)$ is uniformly convex in a neighborhood of $t_0$.
\end{corollary}

\begin{corollary}
    \label{cor:time-convexity-DSL-convexity-short}
    Consider $\dcal := (0,1)\times D$ for some domain $D \subset \RR^{n}$ and let $c\in \left[\frac{n-1}2\pi, \frac{n+1}2\pi\right)$. If $u\in \fcal_c(\dcal)$ then $t\mapsto u(t,x_0)$ is convex for $t\in (0,1)$ for each $x_0 \in D$.
\end{corollary}
\begin{proof}
    If $u\in C^2(\dcal)\cap \fcal_c(\dcal)$, then by Lemma~\ref{lemma:time-convexity-DSL-convexity-short}, $\partial_{t}^2 u \ge 0$ on $\dcal$. Conclude that $u(\cdot, x_0)$ is convex on $(0,1)$. Corollary~\ref{cor:time-convexity-DSL-convexity-short} then follows by Theorem~\ref{thm:restriction-theorem-DSL-convexity-short} and Remark~\ref{remark:pcal-tcal-dirichlet-sets-DSL-convexity-short}.
\end{proof}

\begin{remark}
    The conclusion of Corollary~\ref{cor:time-convexity-DSL-convexity-short} can in fact be extended to boundary points $x_0 \in \partial D$ whenever $u$ extends continuously up to the boundary $\overline{ \dcal}$ (compare with Remark~\ref{remark:boundary-problems-DSL-convexity-short}). If $u\in C^0(\overline{\dcal})$, then for any $x_0 \in \partial D$ $u(\cdot,x_0)$ is a point wise limit of the continuous functions $u(\cdot,x_n)$ for some sequence $x_n \in D$ converging to $x_0$. Thus $t\mapsto u(t, x_0)$ is convex for $t\in (0,1)$, for any $x_0 \in \overline{D}$.
\end{remark}


\subsection{Proof of Theorem \ref{theorem:affine-slices-of-DSL-are-SLag-DSL-convexity-short}}
\label{subsec:proof-of-theorem-1.1-(1)-(2)}

This section contains a proof of Theorem~\ref{theorem:affine-slices-of-DSL-are-SLag-DSL-convexity-short}. Two different proofs of Theorem~\ref{theorem:affine-slices-of-DSL-are-SLag-DSL-convexity-short}(iii) are included. The first is an elementary proof that builds on the ideas of the previous section. The second proof follows by an application of the minimum principle for Lagrangian graphs \cite[Theorem~3.1]{Dar-Rub2019AMP}.

\begin{proof}[Proof of Theorem \ref{theorem:affine-slices-of-DSL-are-SLag-DSL-convexity-short}]
Theorem \ref{theorem:affine-slices-of-DSL-are-SLag-DSL-convexity-short}(i) follows from Corollary~\ref{cor:time-convexity-DSL-convexity-short}, and Theorem \ref{theorem:affine-slices-of-DSL-are-SLag-DSL-convexity-short}(ii) follows from Corollary~\ref{cor:2-convex-second-branch-DSL-convexity-short}. It  remains to prove Theorem \ref{theorem:affine-slices-of-DSL-are-SLag-DSL-convexity-short}(iii). 
    We will give two proofs.
    By assumption $c\ge \frac{n}{2}\pi$, so 
    \beq\lb{cpi2recallEq}
    c-\frac{\pi}{2} \ge \frac{n-1}2\pi.
    \eeq

\noindent
{\it First proof of Theorem \ref{theorem:affine-slices-of-DSL-are-SLag-DSL-convexity-short}(iii).}
    
    \begin{claim}
    \label{claim:affine-slices-are-convex-DSL-convexity-short}
        For any $t_0 \in \RR $ and any ${V} \in \RR^{n\times 1}$, $u\circ \ell_{t_0,{V}} (x) = u(t_0 + {V}^Tx, x)$ is convex on its domain of definition.
    \end{claim}
        \begin{proof}
            By Corollary~\ref{cor:affine-slices-of-DSL-are-SL-DSL-convexity-short}, $u\circ \ell_{t_0,{V}} (x) = u(t_0 + {V}^Tx, x) \in F_{c-\frac{\pi}{2}}$ on its domain of definition. But $F_{c-\frac{\pi}{2}} \subset \pcal_n$ by \eqref{cpi2recallEq} and  Lemma~\ref{lemma:properties-of-eigenvalues-DSL-convexity-short}. Conclude that $u_{t_0,{V}}$ is convex \cite[Proposition~4.5]{HarLaw-dirichletduality}.
        \end{proof}

        Fix $(t_0,x_0), (t_1,x_1)\in \RR\times\RR^n$, and consider an arbitrary line $\RR\ni s\mapsto\gamma(s) = (t_0,x_0)+s(t_1, x_1)\subset\RR^{n+1}$. It suffices to prove that $u$ is convex along $\gamma$ (i.e., that $s\mapsto u\circ\gamma(s)$ is convex) to establish convexity of $u$.
        
        First suppose $x_1 =0$. Then $u(\cdot, x_0)$ is convex by Theorem \ref{theorem:affine-slices-of-DSL-are-SLag-DSL-convexity-short}(iii) and it follows that $u\circ \gamma$ is convex on its domain of definition. Now suppose $x_1\neq 0$. Pick ${V} := t_1x_1^T/|x_1|^2 \in \RR^{n\times 1}$. Then
        \begin{align*}
            u\circ \ell_{t_0 - {V}^Tx_0, {V}} (x_0 +  sx_1) =&\ u(t_0 - {V}^Tx_0 + {V}^T (x_0 +  sx_1), (x_0 +  sx_1) )\\
            =&\ u(t_0 -{V}^Tx_0 + {V}^T x_0 + s t_1, x_0 +  x_1 ) \\
            =&\ u(t_0 + s t_1, x_ 0 + s x_1) = u\circ \gamma(s).
        \end{align*}
        But $u\circ \ell_{t_0 - {V}^Tx_0, {V}} $ is convex by Claim~\ref{claim:affine-slices-are-convex-DSL-convexity-short}
        and the left hand side is the restriction of this function
        to the line $\RR\ni s\mapsto x_0 +  sx_1\in\RR^n$, hence convex in $s$. This concludes the first proof
        of Theorem \ref{theorem:affine-slices-of-DSL-are-SLag-DSL-convexity-short} (iii).

    The second proof of Theorem \ref{theorem:affine-slices-of-DSL-are-SLag-DSL-convexity-short}(iii) is less elementary, as it relies
    on the main theorem of \cite{Dar-Rub2019AMP}, but we
    present it for comparison. In fact, the 
main theorem of \cite{Dar-Rub2019AMP} is generalized in \S\ref{sec:Minimum-principle} using the philosophy of the first proof above.

\vspace{0.5em}
\noindent
{\it Second proof of Theorem \ref{theorem:affine-slices-of-DSL-are-SLag-DSL-convexity-short}(iii).} Recall the partial Legendre transform and envelope defined in \S\ref{sec:Legendre-duality-DSL-convexity-short}. By \cite[Lemma~3.6]{Dar-Rub2019AMP} or by Corollary~\ref{cor:time-convexity-DSL-convexity-short}, $t\mapsto u(t,x)$ is convex for each $x\in D$ and so
\begin{equation}
    \label{eq:2nd-proof-of-thm(iii)-eq1-DSL-convexity-short}
    u(t,x) = u^{\star \star}(t,x) = \sup_{\tau \in \RR } [u^\star(\tau,x) + t\tau]. 
\end{equation}
Then by the minimum principle for Lagrangian graphs in the top branch of the DSL \cite[Theorem~3.1]{Dar-Rub2019AMP},
\begin{equation}
    \label{eq:2nd-proof-of-thm(iii)-eq2-DSL-convexity-short}
    u^\star (\tau, \cdot) = \inf_{t\in [0,1]}[ u(t,\cdot) - t\tau]  \in F_{c-\pi/2}(D), \qquad \forall\ \tau \in \RR.
\end{equation}
Since $c\ge \frac{n}{2} \pi$, $F_{c-\pi/2} \subset \pcal$ by Lemma~\ref{lemma:properties-of-eigenvalues-DSL-convexity-short}, and so $F_{c-\pi/2}(D) \subset \pcal(D)$ \cite[(4.2), p. 409]{HarLaw-dirichletduality}, and $x\mapsto u^\star(\tau,x)$ is convex. Furthermore $(t,x) \mapsto  u^\star(t,x) + t\tau$ is jointly convex in $(t,x)$ being a sum of a convex function in $x$ and a linear function in $t$. Then by \eqref{eq:2nd-proof-of-thm(iii)-eq1-DSL-convexity-short}, $u(t,x)$, being the supremum of a family of jointly convex functions in $(t,x)$, is jointly convex.
\end{proof}

Observe that the partial Legendre transform in time used in 
\cite[Theorem~3.1]{Dar-Rub2019AMP} is {\it not} the standard one
in convexity \cite[p.~104]{rockafeller-convex-analysis-MR1451876}. It is 
instead defined using an infimum in one direction
and a supremum in the other. Showing convexity of the infimum does not follow from simple principles and is the main theorem of 
\cite{Dar-Rub2019AMP}. To go back, convexity is transferred easily via the supremum.

\section{Minimum Principle}
\label{sec:Minimum-principle}

This section establishes the minimum principle for the second branch of the DSL. Definition~\ref{def:product-of-dirichlet-sets-DSL-convexity-short} recalls the Ross--Witt-Nystr\"om $\star$-product. The top two branches of the DSL are endowed with $\star$-product structures in Proposition~\ref{prop:DSL-top-two-branches-are-products-convexity-short-note}, a linear algebra statement whose proof relies on Proposition~\ref{prop:affine-slices-of-DSL-matrix-is-SLag-DSL-convexity-short}, Lemma~\ref{lemma:spectrum-of-AGamma-convexity-note} and Lemma~\ref{lemma:time-convexity-DSL-convexity-short}. Finally Corollary~\ref{cor:minimum-principle-DSL-convexity-short} follows by appealing to the Ross--Witt-Nystr\"om minimum principle for convex $\star$-products \cite{Ross-Witt-Nystrom-minimum-principle-MR4350213}.

\begin{definition}
\label{def:product-of-dirichlet-sets-DSL-convexity-short}
{\rm\cite[Definition~3.1]{Ross-Witt-Nystrom-minimum-principle-MR4350213}} Let $F \subset \Sym^2(\RR^m)$ and $G \subset \Sym^2(\RR^n)$ be two Dirichlet subequations (Definition~\ref{def:dirichlet-set-misc-definitions-DSL-convexity-short}). Recall the pullback $\ell_{t_0,\Gamma}:\RR^n \rightarrow \RR^{m+n}$ defined by \eqref{eq:pull-back-under-ell-DSL-convexity-short}. Let $A\in \Sym^2(\RR^{m+n})$ consider the block representation \eqref{eq:n+m-matrix-decomp-DSL-convexity-short}. Then the $\star$-product Dirichlet subequation $F\star G \subset \Sym^2(\RR^{m+n})$ is the set of matrices $A\in \Sym^2(\RR^{m+n})$ satisfying
\begin{enumerate}[label={(\roman*)}]
    \item $B \in F$,
    \item $\ell_\Gamma^* A \in G$ for all $\Gamma \in \RR^{m\times n}$.
\end{enumerate}
\end{definition}

\begin{remark}
\label{remark:product-convention-DSL-convexity-short}
    The Dirichlet product $F\star G$ is not commutative. The convention in Definition~\ref{def:product-of-dirichlet-sets-DSL-convexity-short} differs from that of Ross--Witt-Nystr\"om~\cite[Definition~3.1]{Ross-Witt-Nystrom-minimum-principle-MR4350213}. Their convention would flip the roles of the two Dirichlet subequations $F$ and $G$, i.e., for them $A\in F\star G$ if
\begin{enumerate}[label={(\roman*)}]
    \item $D \in G$, and
    \item $\displaystyle{
    \begin{pmatrix}
        I & \Gamma^T
    \end{pmatrix} A \begin{pmatrix}
        I \\
        \Gamma
    \end{pmatrix} = B + \Gamma^T C + C^T \Gamma+ \Gamma^T D\Gamma \ \in F \qquad \forall\ \Gamma \in \RR^{n\times m}, 
    }$
\end{enumerate}    
    where $I$ is the $m\times m$ identity matrix. The convention chosen in Definition~\ref{def:product-of-dirichlet-sets-DSL-convexity-short} is consistent with the minimum principle for Lagrangian graphs \cite[Theorem~3.1]{Dar-Rub2019AMP}.
\end{remark}

\begin{proof}[Proof of Proposition~\ref{prop:DSL-top-two-branches-are-products-convexity-short-note}]
    Fix any $c\ge\frac{n-1}2\pi$. Recall the Dirichlet subequations $\pcal_1 = [0,+\infty) \subset \RR = \Sym^2(\RR^1)$ from Definition~\ref{def:dirichlet-set-misc-definitions-DSL-convexity-short}, $F_{c-\frac{\pi}{2}} \subset \Sym^2(\RR^n)$ from \eqref{eq:subeq-for-SL-DSL-convexity-short} and $\fcal_c$ from \eqref{eq:subeq-for-DSL-DSL-convexity-short}. For any $A\in \Sym^2(\RR^{n+1})$, use the block representation \eqref{eq:blockrepresentationofDSLmatrix-DSL-convexity-short}.
    
    \medskip
    \noindent
    \textit{Step (i):} $\fcal_c \subset \pcal_1 \star F_{c-\frac{\pi}{2}}$ for $c\ge \frac{n-1}2\pi$:

        Suppose $A\in \fcal_c$, i.e., $\widetilde{\Theta}(A) \ge c$. Then $a_{00} \ge 0$ by Lemma~\ref{lemma:time-convexity-DSL-convexity-short}, and by Proposition~\ref{prop:affine-slices-of-DSL-matrix-is-SLag-DSL-convexity-short}
        \[\widetilde{\Theta}(A) \ge c \quad \implies \qquad \widetilde{\theta} (\ell_{{V}}^*A) \ge {c-\frac{\pi}{2}} \quad \iff \quad \ell_{{V}}^*A \in F_{c-\frac{\pi}{2}} \qquad \forall\ {V} \in \RR^n.\]
        Conclude by Definition~\ref{def:product-of-dirichlet-sets-DSL-convexity-short} that $A \in \pcal_1 \star F_{c-\frac{\pi}{2}}$. Thus $\fcal_c \subset \pcal_1 \star F_{c-\frac{\pi}{2}}$.
        
    \medskip
    \noindent
    \textit{Step (ii):} $\pcal_1 \star F_{c-\frac{\pi}{2}} \subset \fcal_c $:

        Suppose $A\in \pcal_1 \star F_{c-\frac{\pi}{2}}$. Note that if a dense subset of matrices in $\pcal_1 \star F_{c-\frac{\pi}{2}}$ lie inside $\fcal_c$, then it would follow that $\pcal_1 \star F_{c-\frac{\pi}{2}} \subset \fcal_c$ since both sets are closed. So assume without loss of generality that $a_{00}\neq 0$. By Definition~\ref{def:product-of-dirichlet-sets-DSL-convexity-short}, $a_{00} >0$ and
        \begin{equation}
            \label{eq:DSL-product-main-assumption-DSL-convexity-short}
            \ell_{{V}}^*A \in F_{c-\frac{\pi}{2}} \iff \widetilde{\theta}(\ell_{{V}}^*A) \ge {c-\frac{\pi}{2}} \quad  \forall\ {V} \in \RR^n.
        \end{equation}
         Note that $a_{00}>0$ since by assumption $a_{00} \neq 0$. Fix ${V} \in \RR^{n\times 1}$ such that $a_{00} {V} + \vec{a} = 0$. For such a choice,
         \[A_{{V}} = \diag(a_{00}, \ell_{{V}}^*A)\]
         by \eqref{eq:eq7-DSL-convexity-short}, where $\ell_{{V}}^* A$ is as in \eqref{eq:pull-back-under-ell-DSL-convexity-short-DSL}. Recall by \eqref{eq:eq7-DSL-convexity-short} that $(A_{V})^+ = \ell_{V}^* A$ and so by \eqref{eq:lag-ang-ST-lag-ang-formula-DSL-convexity-short}, Lemma~\ref{lemma:spectrum-of-AGamma-convexity-note} and the specific choice of $V\in \RR^{n\times 1}$,
         \[\begin{aligned}
              \widetilde{\Theta}(A) =&\ \widetilde{\Theta}(A_{V}) = \widetilde{\theta} (\ell_{V}^* A) + \arg(\sqrt{-1}a_{00} + (a_{00} {V}^T + \vec{a}^T )(I+ \sqrt{-1}\ell_{V}^* A)^{-1} (a_{00} {V} + \vec{a} ) ) \\
              =&\ \widetilde{\theta}(\ell_{V}^* A) + \arg(\sqrt{-1}a_{00} ) \ge\ c-\frac{\pi}{2} + \frac{\pi}{2}
         \end{aligned}\]
         where the final inequality follows from \eqref{eq:DSL-product-main-assumption-DSL-convexity-short}. Conclude that $\widetilde{\Theta}(A) \ge c$ or equivalently $A\in \fcal_c$. This proves the converse inclusion $\pcal_1 \star F_{c-\frac{\pi}{2}} \subset \fcal_c$. \qedhere
\end{proof}


\begin{proof}[Proof of Corollary~\ref{cor:minimum-principle-DSL-convexity-short}]
    By Proposition~\ref{prop:DSL-top-two-branches-are-products-convexity-short-note}, $\fcal_c=\pcal_1 \star F_{c-\frac{\pi}{2}}$ for all $c\ge \frac{n-1}2\pi$. The Dirichlet subequation $F_{c-\frac{\pi}{2}}$ is convex when $c\ge \frac{n-1}2\pi$ \cite{yuan-global-solutions-MR2199179}. So the minimum principle of Ross--Witt-Nystr\"om applies to the $\star$-product $\fcal_c = \pcal_1 \star F_{c-\frac{\pi}{2}}$ \cite[Main Theorem]{Ross-Witt-Nystrom-minimum-principle-MR4350213} and so
    \[u \in \fcal_c(\dcal) \quad \implies \quad  v(x) := \inf_{t\in (0,1) } u(t,x) \in  F_{c-\frac{\pi}{2}}(D). \tag*{\qedhere}\]
\end{proof}

\printbibliography

{\noindent\textsc{University of Maryland, College Park, Maryland}

\noindent\href{mailto:pvasanth@umd.edu}{pvasanth@umd.edu}, \href{mailto:yanir@alum.mit.edu}{yanir@alum.mit.edu}}

\end{document}